\documentclass[x11names]{amsart}
\usepackage{amssymb,amsmath,graphicx}
\usepackage{mathrsfs}
\usepackage{color, soul}
\usepackage{array}
\usepackage{longtable}
\usepackage{multicol}
\usepackage{geometry}
\usepackage{subcaption}
\usepackage{enumitem}
\usepackage[bookmarks=true,breaklinks=true]{hyperref}

\newcommand{\xx}{\boldsymbol{x}}
\newcommand{\yy}{\boldsymbol{y}}
\newcommand{\vv}{\boldsymbol{v}}
\newcommand{\ww}{\boldsymbol{w}}
\newcommand{\aalpha}{\boldsymbol{\alpha}}
\newcommand{\RR}{\mathbb{R}}

\newcommand{\mc}{\mathcal}

\newcommand{\VV}{\mathcal{V}}

\newcommand{\rank}{\textup{rank}\,}

\newcommand{\Gr}{\textup{Gr}}

\newcommand{\conv}{\textup{conv}}

\newcommand{\bc}{\color{blue}}

\newtheorem{theorem}{Theorem}

\newtheorem{proposition}[theorem]{Proposition}
\theoremstyle{definition}

\newtheorem{example}[theorem]{Example}

\theoremstyle{remark}
\newtheorem{remark}[theorem]{Remark}

\title{General non-realizability certificates for spheres with linear programming}
\date{}

\author[J. Gouveia]{Jo{\~a}o Gouveia}
\address{CMUC, Department of Mathematics,
  University of Coimbra, 3001-454 Coimbra, Portugal}
\email{jgouveia@mat.uc.pt}

\author[A. Macchia]{Antonio Macchia}
\address{Fachbereich Mathematik und Informatik, Freie Universit\"at Berlin, Arnimallee 2, 14195 Berlin, Germany}
\email{macchia.antonello@gmail.com}

\author[A. Wiebe]{Amy Wiebe}
\address{{Fachbereich Mathematik und Informatik, Freie Universit\"at Berlin, Arnimallee 2, 14195 Berlin, Germany \hspace{420pt} \textcolor{white}{de}
$\phantom{De\,}$Department of Mathematics,
Simon Fraser University,
8888 University Drive,
Burnaby, British Columbia,  V5A 1S6,   Canada}}
\email{w.amy.math@gmail.com}

\thanks{ Gouveia was partially supported by the Centre for Mathematics of the University of Coimbra
- UIDB/00324/2020, funded by the Portuguese Government through FCT/MCTES. Macchia was supported by the Einstein Foundation Berlin under Francisco Santos grant EVF-2015-230 and by the Deutsche Forschungsgemeinschaft (DFG, German
Research Foundation) – project number 454595616. Wiebe was supported by Natural Sciences and Engineering Research Council of Canada (NSERC) [PDF - 557980 - 2021], and by the Pacific Institute for the Mathematical Sciences (PIMS). The research and findings may not reflect those of the Institute.}

\begin{document}
\maketitle

\begin{abstract}
In this paper we present a simple technique to derive certificates of non-realizability for an abstract polytopal sphere. Our approach uses a variant of the classical algebraic certificates introduced by Bokowski and Sturmfels in \cite{CSG}, the final polynomials. More specifically we reduce the problem of finding a realization to that of finding a positive point in a variety and try to find a polynomial with positive coefficients in the generating ideal (a positive polynomial), showing that such point does not exist. Many, if not most, of the techniques for proving non-realizability developed in the last three decades can be seen as following this framework, using more or less elaborate ways of constructing such positive polynomials. Our proposal is more straightforward as we simply use linear programming to exhaustively search for such positive polynomials in the ideal restricted to some linear subspace. Somewhat surprisingly, this elementary strategy yields results that are competitive with more elaborate alternatives, and allows us to derive new examples of non-realizable abstract polytopal spheres.
\end{abstract}


\section{Introduction}

One of the oldest questions in modern polytope theory is whether a given abstract polytopal sphere is realizable as the boundary of a convex polytope. The question was first answered by Steinitz for 3-dimensional polytopes in a theorem which classifies all realizable 3-polytopes in graph-theoretic terms \cite{Steinitz}. To date there is no higher-dimensional analog of Steinitz’s theorem, and attempts to answer the question frequently rely on the theory of oriented matroids, exhaustive computation and classification of spheres with a fixed dimension and number of vertices, and algebraic certificates of non-realizability based on Grassmann-Pl\"ucker relations \cite{BJS90, F20, pfeifle2020positive}.

Due to the large number of Grassmann-Pl\"ucker relations, the search for algebraic certificates based on them---the so-called {\em final polynomials}---often requires some assumption on the structure of such polynomials in order for the search to be feasible.
In this paper, we present a new method to search for algebraic certificates of non-realizability with no assumed structure. Our algorithm uses linear programming together with the more compact description of  the realization space given by the reduced slack ideal model \cite{GMWthirdpaper} to find algebraic certificates of non-realizability. In the end we produce certificates for a collection of large simplicial and quasi-simplicial spheres, including some for which realizability was not previously known.

The rest of the paper is organized as follows. In Section~\ref{sec:alg_certs}, we describe the general techniques used for producing algebraic certificates of non-realizability. These techniques are all based on finding {\em positive polynomials} in an ideal, which we describe how to find for a general ideal in Section~\ref{sec:pos_polys}. Section~\ref{sec:slack_setting} specializes these techniques to the setting of realizability of spheres. This section includes a brief introduction to the reduced slack model of the realization space \cite{GMWthirdpaper}.
Finally in Section~\ref{sec:results}, we discuss the implementation of our algorithm, describe the relation of our certificates to classical final polynomials, and list the results of our computations regarding the realizability of a database of selected spheres. In particular, we derive non-realizability certificates for a large number of new instances of prismatoids, a class of polytopes introduced in \cite{CS19}, recover the recent result of \cite{pfeifle2020positive} on the non-realizability of Jockusch's family of simplicial $3$-spheres (see \cite{NZ20}), as well as providing a few other examples of new or simpler non-realizability certificates.

\section{Algebraic certificates for sphere non-realizability}\label{sec:alg_certs}

In this section we will introduce and contextualize some of the algebraic approaches that have been used to certify that a given sphere is non-realizable.
In order to do that we start by recalling two essentially equivalent models for the realization space of an abstract polytopal sphere: the slack model and the Grassmannian model. In what follows we present only a brief overview of the facts that we will need; a thorough presentation of these models can be found in \cite{GMWthirdpaper}.

Given an abstract $d$-dimensional sphere $P$ with vertex set $\{1,\dots,n\}$ and facet set\break $\mathcal{F}=\{F_1,\dots,F_m\}$, recall that the \emph{symbolic slack matrix} of $P$ is the $n \times m$ matrix $S_P(\xx)$ whose $(i,j)$-entry is zero if $i \in F_j$ and an indeterminate variable otherwise. The \emph{slack variety} of $P$, $\VV_P$, is then the Zariski closure of the set
\[
\{ S_P(\xi) : \rank(S_P(\xi)) \leq d+1 \textrm{ and } \xi \in \RR_*^N\}.
\]
{Throughout the paper we will denote by $\RR_*,\RR_+$, and $\RR_{++}$  the non-zero, non-negative, and positive real numbers respectively. The variety $\VV_P$ is cut} out by the \emph{slack ideal} of $P$, $I_P$, which is the ideal generated by the $(d+2)$-minors of $S_P(\xx)$ saturated by the product of all variables. The slack variety gives us a natural model for the realization space of a polytope up to projective equivalence \cite{GMTWfirstpaper}.

\begin{proposition}[{\cite[Corollary 3.4]{GMTWfirstpaper}}]
There is a one to one correspondence between realizations of a polytope $P$ up to projective equivalence and the elements of $\VV_P \cap \RR^N_{++}$ up to column and row scalings by positive scalars. In particular, $P$ is not realizable if and only if $\VV_P \cap \RR^N_{++} = \emptyset$.
\end{proposition}

This correspondence is explicitly given by the \emph{slack matrices} of each realization. Recall that a (realized) polytope $P$ with vertices $\vv_1,\ldots, \vv_n\in\RR^d$ and facets defined by inequalities $\aalpha_j^\top\xx\leq b_j$, with $\aalpha_j\in\RR^d$ and $b_j \in \RR$, has a slack matrix $S_P$ with entries of the form $(S_P)_{i,j} = b_j - \aalpha_j^\top\vv_i,$ which by construction lies in the slack variety and is well defined up to column scalings.

\medskip
A related  model that has been more classically used to derive non-realizability certificates is the \textit{Grassmannian model}. Let $\Gr(d+1,n)$ be the Grassmannian variety of $(d+1)$-dimensional spaces in $\RR^n$, which we will think of as coordinatized by Pl\"ucker coordinates $\{p_J \, :\, J\subseteq \{1,\dots,n\} \textrm{ and }\break |J|=d+1\}$, and let $I^\Gr(d+1,n)$ be the ideal that cuts out the Grassmannian in those coordinates. We define two special sets of such coordinates. The first, $\Gamma_0$, will be the collection of subsets $J$ such that there is a facet $F$ of $P$ with $J\subset F$. To define the second, $\Gamma_1$, we need to have a way of identifying affine bases of facets of a polytope, i.e., sets of vertices in a facet that are affinely independent in any realization of $P$. We will call such sets {\em facet bases}. Note that in practice, since we do not know anything about the realizations of $P$ a priori, we need a combinatorial way to identify such bases. One way to do this is by choosing a flag in the face lattice. By definition, a \textit{flag} is a maximal chain in the face lattice and hence has length $d+1$ for a $d$-polytope:
\[
\emptyset = G_{-1} \subsetneq G_0 \subsetneq \cdots \subsetneq G_{d-1} \subsetneq G_d = P,
\]
where $G_i$ is an $i$-dimensional face of $P$.
Then a set of vertices chosen so that $\vv_i \in G_i \backslash G_{i-1}$ for $i=0,\ldots, d$ must be affinely independent. (If any $\vv_i$ was in the affine hull of $\vv_0,\ldots, \vv_{i-1}$, then we would have $\vv_i\in G_{i-1}$ by the definition of a face of $P$.)
Thus the vertices indexed by $\{i_0,\ldots, i_{d-1}\}$ chosen as above form a facet basis for $G_{d-1}$. We call the full set of vertices, indexed by $\{i_0,\ldots,i_d\}$, a {\em flag of vertices}. Notice that for any facet $F$ of $P$, there can be many flags with $G_{d-1}=F$, each resulting in a flag of vertices which is the union of a facet basis for $F$ with a vertex not in $F$. To highlight this connection to affine bases of facets, we will also refer to these flags of vertices as {\em facet extensions}, and  $\Gamma_1$  will be the collection of subsets $J$ that are {facet extensions} of $P$. By duality, one can similarly define a \emph{flag of facets} that will be of use later on.

Every coordinate indexed by $J$ in $\Gamma_1$ comes with a sign $\chi_J \in \{\pm 1\}$ depending on the orientation of $J$ as will be discussed later. We then define the \emph{Grassmannian variety of $P$} as the Zariski closure of the set
\[
\Gr(P)=\Pi_{\Gamma_1} \left(
\{ \xi \in \Gr(d+1,n) : \xi_J = 0 \textrm{ for all } J \in \Gamma_0 \textrm{ and } \xi_K \not = 0 \textrm{ for all } K \in \Gamma_1 \} \right),
\]
where $\Pi_{\Gamma_1}$ is the signed projection onto the coordinates in $\Gamma_1$, i.e.,
$\Pi_{\Gamma_1}(\xi)=y \in \RR^{\Gamma_1}$ such that $y_J=\chi_J \xi_J$. Again, the Grassmanian variety gives us a natural model for the realization space of a polytope. In particular, we have a simple realizability characterization.

\begin{proposition}\label{prop:Grassmannian}
$P$ is not realizable if and only if $\Gr(P) \cap \RR^N_{++} = \emptyset$.
\end{proposition}

Note that $\Gr(P)$ and $\VV_P$ are essentially equivalent \cite[Theorem 4.7]{GMWthirdpaper}. More explicitly, for any facet $F_j$ of $P$ pick $B_j$ to be a flag of vertices such that $F_j$ is part of the underlying flag. Then for any $\xi \in \Gr(P)$ the matrix obtained by filling each entry $(i,j)$ of $S_P(x)$ with $\xi_{\{i\} \cup B_j}$ is in $\VV_P$. On the other hand, given any point $\xi$ in the slack variety, the image by $\Pi_{\Gamma_1}$ of the column space of $S(\xi)$ gives us a point in $\Gr(P)$. Moreover, these maps preserve positivity of the coordinates, so it is clear that the realizability questions are totally equivalent, and just offer two possible viewpoints.

In any case, the question of realizability of polytopes boils down to a fundamental question in real algebra: how can we certify that a given variety has no positive points? This is a special case of the more general question of checking emptiness of semialgebraic sets, for which there are several Positivstellensatz type theorems that offer answers. A direct application of the version in \cite{becker1986real} yields the following theorem.

\begin{theorem} \label{thm:certificates}
Given a real variety $\mathcal{V}(I) \subseteq \RR^n$, it has no positive points if and only if
there is an element of $I$ of the form
\[
\xx^{\alpha} + \sum_{i \in I} \xx^{\beta_i} \sigma_i(\xx),
\]
where $\sigma_i(\xx)$ are sums of squares of polynomials.
\end{theorem}

Applied to Proposition \ref{prop:Grassmannian}, the witness polynomials given by this theorem are known as \emph{final polynomials} (see Corollary 4.22 of \cite{CSG}) and have been used since the 1980s for certifying non-realizability of polytopes, see e.g., \cite{bokowski1990finding}, \cite{BJS90}, \cite{RG93}, \cite{FMNR09}, \cite{MP15} and \cite{F20}.
Searching for certificates using the full strength of Theorem \ref{thm:certificates} is possible using semidefinite programming. However, dealing with sums of squares is not always desirable, since semidefinite programming has numerical and scalability issues that are not present in linear programming. A simple alternative is to  consider only scalar~$\sigma_i$.

We will call a polynomial \emph{positive} if it is non-zero and has only non-negative coefficients. Such a polynomial can obviously never vanish in the positive orthant, so we have the following simple proposition, that we can see as a weakening of Theorem \ref{thm:certificates}.

\begin{proposition}
Given a real variety $\mathcal{V}(I) \subseteq \RR^n$, if $I$ has a positive polynomial then
$\mathcal{V}(I)$ has no positive points.
\end{proposition}

This seems like a dramatic weakening of the certificate, and it is fair to ask if it has any use whatsoever in this form. The answer is that it can still be quite effective.
For principal ideals, for example, Polya's Theorem on non-negativity over the simplex \cite{polya1928positive} tells us that {if $\mathcal{V}(\langle p \rangle)$ has no non-negative points, then there is indeed a positive polynomial in $\langle p \rangle$. This is stronger than demanding no positive points, but similar.} In the general case, the picture is not much different. Building on work of Handelman \cite{handelman1985positive}, Einsiedler and Tuncel \cite{ET01} give a full characterization of when positive polynomials exist.

\begin{theorem}
An ideal $I$ has a positive polynomial if and only if for any $\ww \in \RR^N$ the variety of the initial ideal $\textup{in}_{\ww}(I)$ has no positive point.
\end{theorem}

Again, asking that no variety of an initial ideal has a positive point is stronger than simply asking that the ideal has no positive point, but it comes close enough for it to suggest that these types of certificates can still be quite effective. In fact, as we will see, the classical approaches to obtain final polynomials for certifying non-realizability of polytopes have generally relied precisely on constructing positive polynomials in the ideal using combinatorial arguments.

In order to use Proposition \ref{prop:Grassmannian} to effectively build witnesses to non-realizability, there is another issue: getting a good handle on the ideal of $\Gr(P)$. We know that\break $\Gr(P)=\mathcal{V}((I^\Gr(d+1,n) + \langle \Gamma_0 \rangle ) \cap \RR[\Gamma_1])$. However, the ideal $I^\Gr(d+1,n)$ of the Grassmannian is complicated in general; thus, in practical computations, one works with the subideal $I^\textup{tri}(d+1,n)$ generated by the \emph{$3$-term Pl\"ucker relations}, which have the form $x_{ijS}  x_{klS} - x_{ikS}x_{jlS} + x_{ilS}x_{jkS}$, where $S$ is a fixed set of indices different from $i,j,k,l$.

The most popular classic method for proving non-realizability of polytopes is perhaps the special class of final polynomials introduced in \cite{bokowski1990finding}, the \emph{bi-quadratic final polynomials}. These are a special type of positive polynomials in $(I^\textup{tri}(d+1,n) +  \langle \Gamma_0 \rangle ) \cap \RR[\Gamma_1]$ that can be constructed efficiently with linear programming. Another method to construct final polynomials was recently proposed in \cite{pfeifle2020positive}, the \emph{positive Pl\"ucker tree certificates}. Again, these are in fact  positive polynomials in\break $(I^\textup{tri}(d+1,n) +  \langle \Gamma_0 \rangle ) \cap \RR[\Gamma_1]$, constructed with some non-trivial combinatorial reasoning and integer programming.

These methods require additional reasoning or assumptions on the form of the desired positive polynomials because brute force search in the Grassmannian model is limited by how quickly the space of possible polynomials grows. By working directly with the slack ideal model, we see a trade-off in complexity. The ideal generators are generally of higher degree (up to the degree $d+2$ of the minors we start with), but in a smaller number of variables that can be further reduced in many cases by parametrizing the variety.

\section{Finding positive polynomials in general ideals}\label{sec:pos_polys}

Before specializing to slack ideals, we first outline a general approach to finding positive polynomials in general ideals and their parametrized counterparts.

Given an ideal $I=\langle f_1,\dots,f_k\rangle$, which we may assume homogeneous without loss of generality, if we want to check the existence of a positive polynomial in $I$, we want to check if there exists a non-zero polynomial of the form
\[
\sum_{\beta} \xx^{\beta}\left(\sum_{i=1}^k c^{\beta}_i f_i\right),
\]
where the $c^{\beta}_i$ are real numbers, that has only non-negative coefficients. Moreover, if there is any positive polynomial, there is a homogeneous one, so we can fix the degree $D$ and consider the set $I_D$ of all products $\xx^{\beta} f_i$ with $|\beta|=D-\deg(f_i)$. For simplicity suppose $I_D=\{q_1,\dots,q_N\}$. We are then simply searching for a non-zero polynomial $\sum_{i = 1}^N c_i q_i$ with non-negative coefficients. By writing $q_i(\xx) = \sum_{|\alpha|=D} a_i^{\alpha} \xx^{\alpha}$, this simply becomes the linear feasibility question
\begin{equation} \begin{array}{rl}
\textrm{find } c \in \RR^N \textrm{ s.t. } & \displaystyle \sum_{i=1}^N c_i a_i^{\alpha} \geq 0 \textrm{ for all } |\alpha|=D; \\
                                            & \displaystyle\sum_{i=1}^N \sum_{|\alpha|=D} c_i a_i^{\alpha} = 1.
\end{array}\end{equation}
Alternatively, after a little manipulation, we can dualize this problem to the problem of checking if
\begin{equation} \begin{array}{rl}
\textrm{max } \lambda \in \RR \textrm{ s.t. } & \displaystyle \sum_{|\alpha|=D} y_\alpha a_i^{\alpha} = 0 \textrm{ for } i=1,\dots, N; \\
                                            & y_{\alpha} \geq \lambda \textrm{ for all } |\alpha|=D;
\end{array}\end{equation}
is zero or $+ \infty$. We can think of this dual formulation as of asking if the linearization of the system $q_i(\xx)=0$, for $i=1, \dots, N$, attained by replacing each monomial by a distinct variable,  has a solution on the positive orthant.

When trying to apply this method directly to our problem of realizability of spheres, one immediately runs into problems. For $P$ a $d$-dimensional polytope, the usual way to generate the slack variety is by taking all the $(d+2)$-minors of the slack matrix, so $D$ must be at least $d+2$. The number of variables is the number of non-zero entries of the slack matrix, which can be up to $(n-d)m$, where
$n$ is the number of vertices and $m$ the number of facets, and is never much lower than that in the most interesting cases. This means that the number of monomials of degree $D$ in those variables is $\binom{(n-d)m+D-1}{D}$  with $D \geq d+2$. The LP feasibility question we want to solve has, in the primal formulation, this many constrains, and a number of variables that also grows exponentially. Even for polytopes with few vertices and facets in low dimension, this soon gets out of reach for any LP solver. In our case, however, we will see that there is a natural parametrization of the variety associated to the ideal that we can exploit.

Suppose there is a set of variables $\yy=(y_1,\dots,y_m)$ and polynomials $g_1,g_2,\dots,g_N$ such that
\[
\{(g_1({\yy}),\dots, g_N({\yy})) \, : \, \yy \in \RR^m\}
\]
coincides with the variety $\mathcal{V}(I)$. Then positive polynomials on $I$ would immediately translate to relations of the type
\[
\sum_{\beta} c_{\beta}  g^{\beta}(\yy) = 0,
\]
for $c_{\beta}>0$, and where  $g^{\beta}(\yy)$ denotes the product $g_1({\yy})^{\beta_1} g_2({\yy})^{\beta_2} \cdots g_N({\yy})^{\beta_N}$. Reciprocally, any relation of this type will immediately imply that the positive polynomial $p(\xx)=\sum c_{\beta} \xx^\beta$ vanishes on $\mathcal{V}(I)$, and so $\mathcal{V}(I)$ has no positive points (and its vanishing ideal contains a positive polynomial). By limiting our search to a finite set $J$ of $\beta$'s and denoting $g^{\beta}(\yy) = \sum_{\alpha} a_{\beta}^\alpha \yy^{\alpha}$ we transform the search for such certificates into the LP feasibility problem
\begin{equation} \label{eq:primalparametrized}
 \begin{array}{rl}
\textrm{find } c \in \RR^{|J|} \textrm{ s.t. } & \displaystyle \sum_{\beta} c_{\beta} a_{\beta}^{\alpha} = 0 \textrm{ for all } \alpha; \\                                             & \sum_{\beta \in J} c_{\beta}=1;\\
                                            & c_{\beta} \geq 0 \textrm{ for all } \beta \in J.
\end{array}
\end{equation}
Again, by manipulating and dualizing we get the dual formulation
\begin{equation} \label{eq:dualparametrized}
\begin{array}{rl}
\min_{\ww,\lambda} \lambda \in \RR \textrm{ s.t. } & \displaystyle \sum_{\alpha}  a_{\beta}^{\alpha}w_{\beta} \geq - \lambda  \textrm{ for all } \beta,
\end{array}
\end{equation}
which has optimal solution zero if the original formulation is infeasible and $-\infty$ if it is feasible. Note that one can think of this last formulation as simply the standard linearization of the semialgebraic optimization problem

\begin{equation}
\begin{array}{rl}
\min_{\yy,\lambda} \lambda \in \RR \textrm{ s.t. } & g^{\beta}(\yy) \geq - \lambda  \textrm{ for all } \beta.
\end{array}
\end{equation}

\section{Application to realizability of spheres}\label{sec:slack_setting}

Consider a realized polytope $P$ with vertices $\vv_1,\ldots, \vv_n\in\RR^d$, facets defined by inequalities $\aalpha_j^\top\xx\leq b_j$, with $\aalpha_j\in\RR^d$ and $b_j \in \RR$, and  slack matrix $S_P$ with entries $(S_P)_{i,j} = b_j - \aalpha_j^\top\vv_i.$ Now, from \cite{Gouveia2013WhichNM} we know that the rows of $S_P$ form a linearly equivalent realization of $P$. In particular, the rows of any $d+1$ linearly independent columns of $S_P$ also form such a realization. Furthermore, if we have a symbolic slack matrix $S_P(\xx)$,
then we can determine $d+1$ necessarily linearly independent columns by taking a flag of facets. In \cite{GMWthirdpaper}, (if all other facets are simplicial) we call the symbolic slack matrix restricted to these columns a {\em reduced slack matrix}.

Fix one such flag, and let $u_{i}^\top$ be the row corresponding to vertex $i$ in the submatrix of $S_P$ whose columns are those indexed by the flag. Let ${i_1},\ldots, {i_d}$ index an affine basis for facet $F_j$, and consider the linear operator
\[
l_j(x) = \det \begin{bmatrix}  \rule{0.5pt}{8pt} & &\rule{0.5pt}{8pt} & \rule{0.5pt}{8pt} \\[-6pt] u_{i_1} & \cdots & u_{i_d} & x \\  \rule{0.5pt}{8pt}&&\rule{0.5pt}{8pt}&\rule{0.5pt}{8pt} \end{bmatrix}.
\]
This operator vanishes on every $u_i$ such that vertex $i$ is in facet $F_j$, since the columns will be linearly dependent. On the other hand, it does not vanish for every $i$ since the rank of the submatrix of~$S_P$ we are considering is $d+1$. In particular, it will be non-zero whenever a vertex $i$ is not in facet~$F_j$. So we have a linear operator that is zero on the facet and non-zero elsewhere which means that there exists a non-zero real $\lambda_j$ such that for all $i$ we have
\[
l_j(u_i)= \lambda_j (b_j - \aalpha_j^\top\vv_i).
\]
Hence, the matrix $[l_j(u_i)]_{i,j}$ is almost a slack matrix of $P$. The only thing we have to be careful with is the sign of this determinant, for which we have to pay attention to the orientation of $P$.

Given a simplex $\Delta = \conv\{\xx_0,\ldots, \xx_d\}\subseteq\RR^d$,  we can define the {\em orientation} of~$\Delta$ by the sign of
\[
\det \begin{bmatrix} 1 & \cdots & 1 \\ \xx_0 & \cdots & \xx_d \end{bmatrix}.
\]
Since this sign depends on the order of the vertices $\xx_i$, we call an ordering of $0,\ldots, d$ an {\em orientation} of the simplex $\Delta$.

Given an affine basis $B = \{i_0,\ldots, i_{d-1}\}$ for a facet $F$ of a $d$-polytope $P$, an ordering on $B$ determines an orientation of each simplex $B\cup \{v\}$ for $v\in\textup{Vert}(P)\backslash F$ by first taking the ordered elements of $B$ followed by $v$. Since $P$ is a polytope, for a fixed order on $B$, all simplices of this form have the same orientation.

\begin{example} Let $P$ be the triangular prism given by $\conv\{\boldsymbol{0}, e_1, e_2, e_3, e_1+e_3, e_2+e_3\}$. Then a basis for facet $F = \conv\{\boldsymbol{0}, e_1, e_3, e_1+e_3\}$ is given by vertices $1,2$, and $5$ and the sign of the determinants whose columns are indexed by $\{1,2,5,3\}$ and $\{1,2,5,6\}$ is negative:
\[
\det \begin{bmatrix}
1&1&1&1 \\
0&1&1&0 \\
0&0&0&1 \\
0&0&1&0
\end{bmatrix} = \det \begin{bmatrix}
1&1&1&1 \\
0&1&1&0 \\
0&0&0&1 \\
0&0&1&1
\end{bmatrix} =-1.
\]
\label{EX:tripri}
\end{example}

Since $P$ is a polytope, in fact, we can order the vertices of each facet basis so that every  simplex of vertices of $P$ as above has the same orientation. For brevity, we will call a set of facet bases ordered in this way {\em oriented}. Furthermore, we can order the vertices of each facet of $P$ so that the elements of $B$ appear first in the order corresponding to this orientation, followed by the remaining vertices of $F$. When facets are written in this way we will also say they are {\em oriented}.

\begin{example} The triangular prism in Example~\ref{EX:tripri} has facets $123, 456, 1245, 1346, 2356$, and one can check that the set of bases $123, 654, 152, 134, 356$ is oriented. Then $123, 654, 1524, 1346, 3562$ are oriented facets of $P$.
\end{example}

We have defined orientations starting from a realization of $P$ for simplicity. However, one often wishes to determine such an orientation using only the combinatorics of $P$. Even without a realization, we can determine relationships between the orientations of certain facet bases using properties of determinants. For example, if two facets intersect in $d-1$ elements, say $F_1\cap F_2 = \{i_1,\ldots, i_{d-1}\}$, and $j_1\in F_1\backslash F_2$, $j_2\in F_2\backslash F_1$, then we have
\[
\det\begin{bmatrix} \vv_{i_1} & \dots & \vv_{i_{d-1}} & \vv_{j_1} & \vv_{j_2} \end{bmatrix} =  -\det\begin{bmatrix} \vv_{i_1} & \dots & \vv_{i_{d-1}} & \vv_{j_2} & \vv_{j_1} \end{bmatrix},
\]
so that if $ \{i_1,\ldots, i_{d-1},j_k\}$ is a facet basis for $F_k$, $k=1,2$, then they must have opposite orientations. We will talk more about how to determine an orientation later, but will often assume that an orientation is already known.

Returning to our discussion on the linear operator $l_j(x)$, if we assume $P$ is oriented and that $\vv_{i_1},\ldots, \vv_{i_d}$ are always chosen as an oriented affine basis for facet $F_j$, then all the $l_j(u_i)$ will have the same sign, and so the matrix  $[l_j(u_i)]_{i,j}$  is, up to column and row scaling by positive scalars, either $S_P$ or $-S_P$ (depending on the orientation chosen). We can assume to have picked a positive orientation in what follows. We therefore get a parametrization of every entry of the slack matrix as a determinant of the entries in the columns indexed by the flag.

Thus, given a set of oriented facets for $P$, we can construct a parametrization of the slack variety as follows.
\begin{enumerate}
\item Choose a flag $\mathcal{F}$ of facets of $P$.
\item Form the reduced symbolic slack matrix $S_\mathcal{F}(\xx)$. Let $u_i^\top$ denote the  $i$th row of this matrix.
\item Calculate a parametrized slack matrix in the reduced slack variables using:
	\[
    \left(S_P(\xx_\mathcal{F})\right)_{i,F} = \det\begin{bmatrix} u_{F(1)}^\top & \cdots & u_{F(d)}^\top & u_i^\top \end{bmatrix},
    \]
    where $F(1),\ldots, F(d)$ are the elements of the oriented facet basis for $F$.
\end{enumerate}
\begin{remark}
Notice that when $F$ is a non-simplicial facet, there is not a unique choice of vertices for a facet basis. However, since each basis choice defines the same facet, the resulting columns must be linearly dependent. Therefore, in the above parametrization, we may add redundant columns to the reconstructed slack matrix corresponding to the different choices of facet basis. This will not change the properties of the slack matrix, but may add new polynomials to the parametrization and increase the chances of finding a non-realizability certificate (see the last part of Example \ref{E.P3513}).
\end{remark}

Another way to think of the determinants we use to parametrize the slack matrix is as Pl\"ucker coordinates of the reduced slack matrix. That is, if $B = \{F(1),\ldots, F(d)\}$ is the facet basis for $F$,
\[
 \det\begin{bmatrix} u_{F(1)}^\top & \cdots & u_{F(d)}^\top & u_i^\top \end{bmatrix}  = p_{B \cup \{i\}}(S_{\mathcal{F}}(\xx)),
\]
where we recall that the Pl\"ucker coordinates of a rank $d+1$ matrix $A\in\RR^{n\times (d+1)}$ are indexed by sets of $d+1$ rows $p_{i_0,\ldots, i_d}(A)$. This is the classic parametrization of the Grassmannian and specializes to a natural parametrization of the variety $\Gr(P)$. On the other hand, this is not truly a parametrization of the entirety of the slack variety, as we are not free to scale columns, but it gives us at least an element per equivalence class, modulo positive column scalings, which is enough for our purposes.

\begin{example} \label{ex:prism}
Consider the triangular prism with oriented facets $132, 645, 1254, 1436, 3652$. The last four facets form a flag of facets whose reduced slack matrix is
\[
S_\mathcal{F}(\xx) = \begin{bmatrix}
x_{1,1} & 0 & 0 & x_{1,4} \\
x_{2,1} & 0 & x_{2,3} & 0 \\
x_{3,1} & x_{3,2} & 0 & 0 \\
0 & 0 & 0 & x_{4,4} \\
0 & 0 & x_{5,3} & 0 \\
0 & x_{6,2} & 0 & 0
\end{bmatrix}
\]
and the parametrized slack matrix is
\[
S_P(\xx_{\mathcal{F}}) =
\begin{small}
\begin{bmatrix}
0 & x_{1,1}x_{4,4}x_{5,3}x_{6,2} & 0 & 0 & x_{1,4}x_{3,1}x_{5,3}x_{6,2} \\
0 & x_{2,1}x_{4,4}x_{5,3}x_{6,2} & 0 & x_{1,1}x_{2,3}x_{3,2}x_{4,4} & 0 \\
0 & x_{3,1}x_{4,4}x_{5,3}x_{6,2} & x_{1,4}x_{2,1}x_{3,2}x_{5,3} & 0 & 0 \\
x_{1,1}x_{2,3}x_{3,2}x_{4,4} & 0 & 0 & 0 & x_{3,1}x_{4,4}x_{5,3}x_{6,2} \\
x_{1,4}x_{2,1}x_{3,2}x_{5,3} & 0 & 0 & x_{1,1}x_{3,2}x_{4,4}x_{5,3} & 0 \\
x_{1,4}x_{2,3}x_{3,1}x_{6,2} & 0 & x_{1,4}x_{2,1}x_{5,3}x_{6,2} & 0 & 0
\end{bmatrix}.
\end{small}
\]
In this case, we can directly see that the parametrized slack variety contains a realization (a point for which all non-zero entries of the slack matrix have the same sign, namely when we set all variables to 1).
\end{example}

We are now ready to search for certificates using the linear program \eqref{eq:dualparametrized}, since we have a parametrization of $\Gr(P)$ by  polynomials of degree $d+1$ in the variables of the reduced symbolic slack matrix. In order to further reduce the computational  effort, we will however need some further considerations.

Note that the slack variety is invariant under scalings of rows and columns. This means that we can always scale rows and columns by positive scalars as to fix some entries to be one without loss of generality (see \cite[Lemma 5.2]{GMTWsecondpaper}): if there was a positive point in the variety before there is still one now. The same is true for the parametrized version: if we scale the rows and columns of the reduced slack matrix to fix some entries to be one, we obtain a dehomogenized version of the parametrized variety for which the same tools as before can be used. A more rigorous discussion of the scaling procedure can be found in \cite[Section 3]{MW20}, but we can see below some examples of the procedure.

\begin{example} Recall the triangular prism $P$ from Example \ref{ex:prism} with the same orderings and labellings, and the same reduced slack matrix. We can scale the first three rows to set $x_{1,1}, x_{2,1}$ and $x_{3,1}$ to be one, then the last three columns to set $x_{3,2}$, $x_{2,3}$ and $x_{1,4}$ to also be one, and finally the last three rows to set the remaining variables to one. In fact, we can eliminate all variables, obtaining a single point in the parametrized variety, which is
\[
\begin{bmatrix}
0 & 1 & 0 & 0 & 1 \\
0 & 1 & 0 & 1 & 0 \\
0 & 1 & 1 & 0 & 0 \\
1 & 0 & 0 & 0 & 1 \\
1 & 0 & 0 & 1 & 0 \\
1 & 0 & 1 & 0 & 0
\end{bmatrix}
\]
and has only positive signs outside of the forced zeroes. Thus, $P$ is realizable. Note that the reason we can have such dramatic reduction of the dimension of the slack variety is that the triangular prism is projectively unique.
\end{example}

We will denote by $S^H(\xx)$ the original, homogenous, parametrized slack matrix, and by $S(\xx)$ the dehomogenized version. In most circumstances we will drop the variables from the notation and simply use $S^H$ and $S$ if there is no contextual ambiguity. We will use the polynomials $S_{i,j}$, the entries of the dehomogenized slack matrix, as the polynomials for applying the schemes \eqref{eq:primalparametrized} or \eqref{eq:dualparametrized}.

\section{Computational results}\label{sec:results}

\subsection{General framework}

Given a candidate abstract polytope $P$, for which we have computed  the dehomogenized parametrized slack matrix $S$ as described in the previous section, we will proceed as follows. For fixed positive integers $k,l$ we will construct the set $\mc G_{k,l}$ of all products of at most $k$ entries (possibly repeated) of $S$, each entry with degree at most $l$. Let $\mc G_{k,l}=\{g_1,\dots,g_m\}$ and $g_i(\xx) = \sum_\alpha a^{\alpha}_i \xx^{\alpha}$; then we will solve the slight modification of \eqref{eq:primalparametrized} given by
\begin{equation} \label{eq:primalparametrizedv2}
\begin{array}{rl}
\displaystyle \textrm{min } 1- \sum_{i=1}^m c_i \textrm{\quad s.t. } & \displaystyle \sum_i c_i a_i^{\alpha} = 0 \textrm{ for all } \alpha; \\
 & \displaystyle \sum_{i=1}^m c_{i} \leq 1;\\
 & \displaystyle c_i \geq 0 \textrm{ for all } i=1,\dots,m.
\end{array}
\end{equation}
This problem is always feasible, and it is a simple exercise to check that its optimal value will be either $0$ or $1$. If it is $0$, that means that a certificate of non-realizability of the polytope was found, and one can write it explicitly as
\[
\sum_{i=1}^m c_i g_i(\xx) = 0,
\]
which cannot happen in a realizable polytope as all entries of the slack matrix must be strictly positive. We will call a certificate obtained in this way from \eqref{eq:primalparametrizedv2} using~$\mc G_{k,l}$ a \emph{$(k,l)$-positive polynomial}.

\begin{remark} Note that even for a fixed set $\mc G_{k,l}$, the certificates we obtain from \eqref{eq:primalparametrizedv2} are in general not unique and will depend, for example, on which algorithm is used to solve the linear program.  The examples below were obtained using the primal simplex method implemented in \textit{Gurobi $9.1.0$}. While Gurobi is an inexact solver, in all the examples we computed all recovered coefficients were integers (in fact, plus or minus one). Furthermore, we always double-check that the certificates sum to zero using the reconstructed slack matrix, i.e., symbolic computations, which is a cheap exact verification of the certificate obtained. In case the inexactness causes problems in particular instances, one could resort to an exact LP solver.
\end{remark}

In what follows we explore two concrete examples of non-realizable spheres, to explicitly illustrate how these certificates work.

\begin{example}\label{E.N10_3574}
Let $P$ be the $4$-dimensional simplicial sphere $N_{3574}^{10}$ in \cite{A77} with $10$ vertices and $35$ facets. This was shown to be non-realizable by Joswig and R\"orig in \cite[Remark 2.4 and page 227]{JR07}. We recover this result with our algorithm. Under some vertex labeling, we have a flag
\[
\mc F = \{F_1=\{3, 4, 8, 9\}, F_2=\{3, 5, 9, 8\}, F_3=\{2, 3, 7, 8\}, F_4=\{3, 4, 9, 6\}, F_5=\{4, 6, 10, 9\}\},
\]
where we ordered the vertices in each facet so that the facet is positively oriented.  We consider the corresponding reduced slack matrix with the following dehomogenization:
\begin{small}
\[
S_{\mc F}(\xx) = \begin{bmatrix}
x_{1,1} & x_{1,2} & x_{1,3} & 1 & x_{1,5} \\
x_{2,1} & x_{2,2} & 0 & 1 & x_{2,5} \\
0 & 0 & 0 & 0 & 1  \\
0 & 1 & x_{4,3} & 0 & 0  \\
x_{5,1} & 0 & x_{5,3} & 1 & x_{5,5} \\
x_{6,1} & 1 & x_{6,3} & 0 & 0  \\
1 & 1 & 0 & 1 & 1  \\
0 & 0 & 0 & 1 & x_{8,5} \\
0 & 0 & 1 & 0 & 0  \\
x_{10,1} & x_{10,2} & 1 & 1 & 0
\end{bmatrix}.
\]
\end{small}
Searching for $(2,2)$-positive polynomial certificates, we find the certificate of non-realizability
\[
S_{9,8} S_{8,4} + S_{9,8} S_{5,7} + S_{10,7} S_{3,10} + S_{9,11} S_{8,4} + S_{3,9} S_{9,6} = 0,
\]
where the facets outside the flag that appear are, ordered positively,
\begin{gather*}
F_6=\{1, 3, 7, 6\}, F_7=\{3, 6, 9, 7\}, F_8=\{1, 3, 6, 10\},\\
F_9=\{4, 5, 9, 10\}, F_{10}=\{1, 4, 5, 9\} \textrm{ and  } F_{11}=\{1, 3, 10, 7\}.
\end{gather*}
\end{example}

\begin{example} \label{ex:Doolittle} In a private communication \cite{D21}, Joseph Doolittle provided us with three simplicial spheres obtained by  subdividing some facets of $N_{3574}^{10}$ (see Example \ref{E.N10_3574}). Here we consider one such simplicial sphere $P$ with $13$ vertices and $65$ facets. Under some vertex labeling, we have a flag
\[
\mc F = \{F_1=\{3, 4, 7, 11\}, F_2=\{3, 7, 13, 11\}, F_3=\{3, 4, 11, 10\}, F_4=\{1, 2, 8, 7\}, F_5=\{3, 7, 12, 13\}\},
\]
where we ordered the vertices in each facet so that the facet is positively oriented.  We consider the corresponding reduced slack matrix with the following dehomogenization:
\begin{small}
\[
S_{\mc F}(\xx) = \begin{bmatrix}
x_{1,1} & x_{1,2} & x_{1,3} & 0 & 1 \\
x_{2,1} & x_{2,2} & x_{2,3} & 0 & 1 \\
0 & 0 & 0 & 1 & 0 \\
0 & x_{4,2} & 0 & x_{4,4} & 1 \\
x_{5,1} & x_{5,2} & x_{5,3} & x_{5,4} & 1 \\
1 & 1 & 1 & 1 & 1 \\
0 & 0 & 1 & 0 & 0 \\
x_{8,1} & x_{8,2} & x_{8,3} & 0 & 1 \\
x_{9,1} & x_{9,2} & x_{9,3} & x_{9,4} & 1 \\
x_{10,1} & x_{10,2} & 0 & x_{10,4} & 1 \\
0 & 0 & 0 & x_{11,4} & 1 \\
1 & x_{12,2} & x_{12,3} & x_{12,4} & 0 \\
1 & 0 & x_{13,3} & x_{13,4} & 0
\end{bmatrix}.
\]
\end{small}
Searching again for $(2,2)$-positive polynomial certificates, we find the certificate of non-realizability
\begin{gather*}
S_{7,6} + S_{3,7} + S_{3,10} + S_{3,14} + S_{7,9} + S_{3,8} + S_{7,13} S_{3,15} + S_{3,12} S_{7,13} + S_{3,4} + S_{7,11} S_{3,15} = 0,
\end{gather*}
where the facets outside the flag that appear are
\begin{gather*}
F_6=\{2, 3, 6, 10\}, F_7=\{1, 2, 7, 5\}, F_8=\{1, 6, 10, 7\}, F_9=\{2, 3, 8, 6\}, F_{10}=\{2, 5, 10, 7\},\\
F_{11}=\{3, 6, 11, 8\}, F_{12}=\{1, 7, 8, 13\}, F_{13}=\{1, 3, 6, 11\}, F_{14}=\{1, 5, 7, 10\} \textrm{ and } F_{15}=\{1, 6, 7, 13\}.
\end{gather*}
\end{example}

\subsection{Recovering classic final polynomials}

While the certificates in the previous examples prove non-realizability independently, one might wish to derive classical final polynomial certificates in Pl\"ucker coordinates from them. As we discussed previously, the parametrization we are using for the slack matrices is, in fact, a parametrization of the Grassmannian variety, hence it should be almost automatic to switch from one to the other.
However, to translate our parametrization back to Pl\"ucker coordinates, we need to undo the scaling of the entries of the reduced slack matrix that we set to one. We will see that this rehomogenization step can sometimes complicate the translation of our certificates.

\begin{example}
In Example \ref{E.N10_3574} we derived the certificate
\[
S_{9,8} S_{8,4} + S_{9,8} S_{5,7} + S_{10,7} S_{3,10} + S_{9,11} S_{8,4} + S_{3,9} S_{9,6} = 0.
\]
If we denote by $S^H(\xx)$ the rehomogenized reconstructed slack matrix, {then in the certificate in the rehomogenized variables, we need to introduce extra variables to maintain the equality}, thus obtaining the homogeneous certificate
\begin{gather*}
S_{9,8}^H S_{8,4}^H {\bc x_{5,4}x_{7,2}} + S_{9,8}^H S_{5,7}^H {\bc x_{4,2}x_{8,4}} + S_{10,7}^H S_{3,10}^H {\bc x_{6,2}x_{8,4}} + S_{9,11}^H S_{8,4}^H {\bc x_{5,4}x_{6,2}} + S_{3,9}^H S_{9,6}^H {\bc x_{6,2}x_{8,4}} = 0.
\end{gather*}
Note that we multiplied each term by one variable from the second column and one from the fourth column of the reduced slack matrix. On the other hand, if $F_j$ is in the flag, then $S_{i,j}^H$ is a multiple of $x_{i,j}$, with the extra factor depending only on $j$. This means that we can replace such variables by slack entries maintaining the validity of the certificate
\begin{gather*}
S_{9,8}^H S_{8,4}^H S^H_{5,4}S^H_{7,2} + S_{9,8}^H S_{5,7}^H S^H_{4,2}S^H_{8,4} + S_{10,7}^H S_{3,10}^H S^H_{6,2}S^H_{8,4} + S_{9,11}^H S_{8,4}^H S^H_{5,4}S^H_{6,2} + S_{3,9}^H S_{9,6}^H S^H_{6,2}S^H_{8,4} = 0.
\end{gather*}
Factoring out $S_{8,4}^H$, we get
\begin{gather*}
S_{9,8}^H  S^H_{5,4}S^H_{7,2} + S_{9,8}^H S_{5,7}^H S^H_{4,2} + S_{10,7}^H S_{3,10}^H S^H_{6,2} +
S_{9,11}^H  S^H_{5,4}S^H_{6,2} + S_{3,9}^H S_{9,6}^H S^H_{6,2} = 0.
\end{gather*}
This can now be immediately translated into a final polynomial: \begin{align*}
p_{1, 3, 6, 10, 9}p_{3, 4, 9, 6, 5}p_{3, 5, 9, 8, 7}
+p_{1, 3, 6, 10, 9}p_{3, 6, 9, 7, 5}p_{3, 5, 9, 8, 4}
+p_{3, 6, 9, 7, 10}p_{1, 4, 5, 9, 3}p_{3, 5, 9, 8, 6}& \\
+\,p_{1, 3, 10, 7, 9}p_{3, 4, 9, 6, 5}p_{3, 5, 9, 8, 6}
+p_{4, 5, 9, 10, 3}p_{1, 3, 7, 6, 9}p_{3, 5, 9, 8, 6} &= 0,
\end{align*}
where each ordering of the Pl\"ucker coordinate is positive because it can be interpreted as the evaluation of a facet inequality on a vertex outside that facet.

{Once we have such a certificate, an alternative explanation for its validity which avoids our slack matrix framework is to prove that it is indeed a (positive) polynomial in the Grassmannian ideal. }
In this case, we observe that this polynomial can be written as the sum of the following polynomials
\begin{align*}
 & p_{1, 3, 6, 10, 9} ( -[3,5,9,4\, |\, 3,5,9,6,7,8] - p_{3,5,9,6,8}p_{3,5,9,4,7}) \\
 & p_{3, 5, 9, 8, 6} ( -[3,4,5,9\,|\,1,3,6,7,9,10] + p_{3,4,5,9,7}p_{1,3,6,9,10} ),
\end{align*}
where
$[i_1,\ldots,i_{d} \,|\,  j_1, \ldots, j_{d+2}]$ denotes the Pl\"ucker relation
\[
\sum_{k=1}^{d+2}(-1)^k p_{i_1,\ldots,i_d,j_k}p_{j_1,\ldots, \widehat{j_k},\ldots, j_{d+2}} = 0.
\]
Thus, the polynomial reduces to the following expression which is indeed zero:
\[
-p_{1, 3, 6, 10, 9}p_{3,5,9,6,8}p_{3,5,9,4,7} +
p_{3, 5, 9, 8, 6}p_{3,4,5,9,7}p_{1,3,6,9,10}.
\]
\end{example}

Transforming our certificates into traditional final polynomials might in some cases create much more complicated certificates.

\begin{example}
In Example \ref{ex:Doolittle} we derived the certificate
\begin{gather*}
S_{7,6} + S_{3,7} + S_{3,10} + S_{3,14} + S_{7,9} + S_{3,8} + S_{7,13} S_{3,15} + S_{3,12} S_{7,13} + S_{3,4} + S_{7,11} S_{3,15} = 0.
\end{gather*}
In this case the rehomogenized version looks more complicated:
\begin{gather*}
S_{7,6}^H {\bc x_{1,5}^2 x_{3,4} x_{5,5} x_{6,2} x_{7,3} x_{8,5} x_{11,5}x_{13,1}} + S_{3,7}^H {\bc x_{1,5}x_{3,4}x_{6,2}x_{6,5}x_{7,3}x_{8,5}x_{10,5} x_{11,5}x_{13,1}} +\\
S_{3,10}^H {\bc x_{1,5}^2 x_{3,4}x_{6,2}x_{6,5}x_{7,3}x_{8,5}x_{11,5}x_{13,1}} + S_{3,14}^H {\bc x_{1,5}x_{2,5}x_{3,4}x_{6,2}x_{6,5}x_{7,3}x_{8,5} x_{11,5}x_{13,1}} + \\
S_{7,9}^H {\bc x_{1,5}^2 x_{3,4}x_{5,5}x_{6,2}x_{7,3}x_{10,5}x_{11,5}x_{13,1}} + S_{3,8}^H {\bc x_{1,5}x_{2,5}x_{3,4}x_{5,5}x_{6,2}x_{7,3}x_{8,5} x_{11,5}x_{13,1}} + \\
S_{7,13}^H S_{3,15}^H {\bc x_{2,5}x_{5,5}x_{8,5}x_{10,5}} + S_{3,12}^H S_{7,13}^H {\bc x_{2,5}x_{5,5}x_{6,5}x_{10,5}} + \\
S_{3,4}^H {\bc x_{1,5}x_{3,4}x_{5,5}x_{6,2}x_{6,5}x_{7,3}x_{10,5} x_{11,5}x_{13,1}} + S_{7,11}^H S_{3,15}^H {\bc x_{1,5}x_{2,5}x_{5,5}x_{10,5}} = 0.
\end{gather*}
Noting that $S^H_{6,2}=x_{13,1}x_{6,2}x_{7,3}x_{3,4}x_{11,5}$, we can simplify this to
\begin{gather*}
S_{7,6}^H S_{6,2}^H{\bc x_{1,5}^2 x_{5,5} x_{8,5}} + S_{3,7}^H S_{6,2}^H {\bc x_{1,5}x_{6,5}x_{8,5}x_{10,5}} + S_{3,10}^H S_{6,2}^H{\bc x_{1,5}^2 x_{6,5}x_{8,5}} + S_{3,14}^H S_{6,2}^H{\bc x_{1,5}x_{2,5}x_{6,5}x_{8,5}} +\\
S_{7,9}^H S_{6,2}^H{\bc x_{1,5}^2 x_{5,5}x_{10,5}} + S_{3,8}^HS_{6,2}^H {\bc x_{1,5}x_{2,5}x_{5,5}x_{8,5}} + S_{7,13}^H S_{3,15}^H {\bc x_{2,5}x_{5,5}x_{8,5}x_{10,5}} +\\
S_{3,12}^H S_{7,13}^H {\bc x_{2,5}x_{5,5}x_{6,5}x_{10,5}} + S_{3,4}^H S_{6,2}^H {\bc x_{1,5}x_{5,5}x_{6,5}x_{10,5}} + S_{7,11}^H S_{3,15}^H {\bc x_{1,5}x_{2,5}x_{5,5}x_{10,5}} = 0.
\end{gather*}
Since there are four variables in each term, and all from column $5$, we can replace each variable by the corresponding homogeneous entry, and then translate it to Pl\"ucker coordinates to get the following degree six final polynomial, which is much more complicated than the original certificate:
\begin{gather*}
p_{2,3,6,10,7} p_{3,7,13,11,6} { p_{3,7,12,13,1}^2 p_{3,7,12,13,5}p_{3,7,12,13,8} } \\
+\,p_{1,2,7,5,3}  p_{3,7,13,11,6} { p_{3,7,12,13,1}p_{3,7,12,13,6}p_{3,7,12,13,8}p_{3,7,12,13,10} } \\
+\,p_{2,5,10,7,3} p_{3,7,13,11,6} { p_{3,7,12,13,1}^2 p_{3,7,12,13,6}p_{3,7,12,13,8} } \\
+\,p_{1,5,7,10,3} p_{3,7,13,11,6} { p_{3,7,12,13,1}p_{3,7,12,13,2}p_{3,7,12,13,6}p_{3,7,12,13,8} } \\
+\,p_{2,3,8,6,7}  p_{3,7,13,11,6} { p_{3,7,12,13,1}^2 p_{3,7,12,13,5}p_{3,7,12,13,10} } \\
+\,p_{1,6,10,7,3} p_{3,7,13,11,6} { p_{3,7,12,13,1}p_{3,7,12,13,2}p_{3,7,12,13,5}p_{3,7,12,13,8}  } \\
+\,p_{1,3,6,11,7} p_{1,6,7,13,3}  { p_{3,7,12,13,2}p_{3,7,12,13,5}p_{3,7,12,13,8}p_{3,7,12,13,10} } \\
+\,p_{1,7,8,13,3} p_{1,3,6,11,7}  { p_{3,7,12,13,2}p_{3,7,12,13,5}p_{3,7,12,13,6}p_{3,7,12,13,10} } \\
+\,p_{1,2,8,7,3}  p_{3,7,13,11,6} { p_{3,7,12,13,1}p_{3,7,12,13,5}p_{3,7,12,13,6}p_{3,7,12,13,10} } \\
+\,p_{3,6,11,8,7} p_{1,6,7,13,3}  { p_{3,7,12,13,1}p_{3,7,12,13,2}p_{3,7,12,13,5}p_{3,7,12,13,10} } = 0.
\end{gather*}
Note that there is more than one way to make this rehomogenization, so there could conceivably be easier certificates that can be derived in this way. Proving that this is indeed a final polynomial independently of our previous computations is possible, but not immediate.
\end{example}

\subsection{Implementation details}
In the beginning of the section we established our general approach, which is essentially that of solving an instance of problem \eqref{eq:primalparametrizedv2}. In this subsection we discuss additional details of the actual instantiation, explaining our implementation.

We describe the main steps in our algorithm for the search of non-realizability certificates. Let~$P$ be an abstract $d$-dimensional polytope with $n$ vertices and $m$ facets.
The inputs of our algorithm are:
\begin{itemize}
\item list of facets of $P$ given as lists of vertex labels;
\item $d$: dimension of abstract polytope $P$;
\item $\mc F$: flag as list of $d+1$ facet indices;
\item $k$: maximum number of factors in the products of constraints;
\item $l$: maximum degree of constraints to consider.
\end{itemize}
Note that the computation of a flag $\mc F$ can be done automatically from the list of facets, using the \textit{SlackIdeals} package in \textit{Macaulay2} \cite{MW20}. Our implementation uses \textit{SageMath} \cite{Sagemath}, that includes \textit{Macaulay2} and the LP solver \textit{Gurobi $9.1.0$} \cite{Gurobi}.

\subsubsection{Step 1: Constructing parametrization}

If we have a given orientation of the facets of~$P$, we can construct each parametrized entry $S_{i,j}$ of the slack matrix by simply computing the corresponding determinant, as explained in Section \ref{sec:slack_setting}.

An orientation can be computed using methods implemented in \textit{Polymake} \cite{Polymake}. Alternatively, the method we describe below orients the facets at the same time as constructing the parametrization.  While this method is not guaranteed to find an orientation, it always succeeded in our tests.

\begin{enumerate}
\item Choose a basis $B_j$ for each (non-simplicial) facet $F_j$ of $P$.
\item Use the given flag  $\mc F$ to form the corresponding dehomogenized reduced slack matrix $S_{\mc F}(\xx)$.
\item Reconstruct all entries of the slack matrix (including the entries corresponding to the facets in $\mc F$) via the appropriate determinants of $S_{\mc F}(\xx)$.
\end{enumerate}
If we were given an orientation, then we can compute the determinants so that all entries have the appropriate sign. Otherwise we proceed to simultaneously find an orientation and the correct sign of each entry of the matrix with the algorithm below.

\begin{enumerate}[resume]
\item Initialize the set $\mathcal{P}$ of known positive polynomials to be the set of all variables.
\item Check all columns whose sign is unknown for entries that are of the form {$\pm\prod q_i$}, where each $q_i\in\mathcal{P}$. If such a polynomial exists it determines the sign of that column, since $\prod q_i > 0$. Add the entries of that column, with any monomial factors removed, with the correct sign to $\mathcal{P}$.
\item We repeat the last step until the sign of every column has been determined or the set of positive polynomials is unchanged.

\end{enumerate}

\subsubsection{Step 2: Constructing constraints and their products for linear program}

Given the previously determined orientation, we have a slack matrix whose entries are Pl\"ucker coordinates with the correct sign. In particular, the set of entries and any products thereof gives us a collection of polynomials that must be positive. We choose a set of constraints by restricting to entries of degree up to $l$ and taking products of up to $k$ of these entries. Call this set of constraints $\mc G_{k,l}$.

We store these products as a matrix $\mc M_{k,l}$ of coefficients, where each row represents a constraint in $\mc G_{k,l}$ and each column is a monomial that appears in some constraint. By storing only the coefficients of the (distinct) monomials, this matrix effectively records the linearization of our constraints. Thus we now have the constraints of a linear program.

\subsubsection{Step 3: Solving linear program with Gurobi}

\begin{enumerate}
\item We solve the linear program \eqref{eq:primalparametrizedv2} whose coefficient matrix is the transpose of $\mc M_{k,l}$ using the primal simplex method in \textit{Gurobi 9.1.0}.
\item If the optimal solution is zero, then we find the indices of the non-zero dual variables, which correspond to an infeasible set $\mc O$ of primal constraints.
\end{enumerate}

\subsubsection{Step 4: (Optional) Rehomogenizing infeasibility certificate}

As we saw in the previous section, it is possible, but not necessary, to rehomogenize the certificates attained in the last step. That is done with the following method.

\begin{enumerate}
\item As in Step 1 (2)--(3), use the given flag $\mc F$ to form the corresponding (homogeneous) reduced slack matrix $S_{\mc F}^H(\xx)$. Reconstruct all entries of the slack matrix (including the entries in the facets in $\mc F$) via the appropriate determinants of $S_{\mc F}^H(\xx)$ and using the orientation of facets determined at Step 1. (Unlike Step 1, here we do not dehomogenize $S_{\mc F}(\xx)$).
\item {We recompute the entries of the reconstructed slack matrix $S^H(\xx)$ that correspond to entries used in the certificate $\mc O$ and then identify the variables we need to multiply each entry by to maintain the certificate validity; that is, we want}
    \[
    \sum_{q \in \mc O} \beta_q q = 0,
    \]
    where the polynomials $q$ are entries of $S^H(\xx)$ and $\beta_q$ are monomial factors.
\end{enumerate}

\subsection{Constraint selection heuristics} \label{S.constraintSelection}

While the proposed approach yields results in several interesting cases, its computational difficulty grows quickly with the number of vertices and facets. This motivates the need for heuristic techniques to reduce the size of the problem when trying to tackle polytopes whose slack matrices are too big for the full strength of our proposed approach. We provide three such heuristics, that can be used individually or combined, in order to derive certificates for larger problems.
\medskip

{\bf Vertex avoidance. \!} Examining the certificates obtained using the full power of our methods, certain trends can be observed. Take the certificate of Example \ref{ex:Doolittle}, where we can see that none of the slack entries used corresponds to any of the vertices $\{4,9,12\}$ or any facet containing at least one of them. These three vertices actually form a triangular face of this polytope.
This type of behaviour, where the obtained certificates avoid the rows indexed by the vertices of a fixed face and the columns indexed by any facet that intersects this face, seems to be extremely common, and motivates our first heuristic simplification: pick a face $G$ of the polytope and remove from the parametrized slack matrix the rows of the vertices of $G$ and the columns of facets intersecting~$G$.
\medskip

{\bf Vertex fixing. \!} If, instead of absent vertices, we focus on vertices that do appear in the certificate, another pattern emerges. Again looking at Example \ref{ex:Doolittle} we can now see that if we take any entry $(i,j)$ of the slack matrix appearing in that certificate and take the union of the vertex $i$ and the vertices of the facet $F_j$, it always contains both vertices $3$ and $7$, which are the vertices of an edge of the polytope. Again, this is a behaviour that can be repeatedly observed in our numerical testing, suggesting a second heuristic simplification: pick a face $G$ of the polytope and consider only entries $(i,j)$ of the parametrized slack matrix such that $\{i\}\cup F_j$ contains $G$. \medskip

While we formulate and use the previous two heuristics in terms of faces, it might be useful in some cases,  particularly in polytopes with a large number of vertices, to avoid or fix a more general set of vertices, not necessarily forming a face.

\medskip
{\bf Monomial simplification. \!}  A somewhat different reduction that can be done is an algebraic simplification. When computing the parametrized slack matrix, the reconstructed entries of a column in the flag are simply the original entries of the corresponding column of the reduced slack matrix (whose entries are either $0$, $1$ or a single variable) all multiplied by the same polynomial, which is the $d$-minor indexed by the remaining columns of the flag and the vertices of the chosen facet basis of $F$. This minor has to be non-negative, and can be added to the set of polynomial constraints. In fact, quite often there is more than one variable that can be factored out in each entry, maintaining the positivity of the remaining expression. This means that when restricting the constraints to be used by degree, we may attain a richer set at a lower degree, thus obtaining solutions at easier computational regimes. Adding these minors is our third and last proposed heuristic and its effect can be seen at the end of Example \ref{E.P3513}. If any of the facets in the chosen flag is not a simplex, we will also add redundant columns with all possible choices of facet bases, and perform this factorization for each of them, as that can lead to different low degree polynomials in our constraint set, strengthening its expressive power.

\medskip

\begin{example} \label{E.P3513}
Let $P$ be prismatoid \#3513 from \cite{CS19}. It is $5$-dimensional and has $14$ vertices and $94$ facets, two of which are non-simplicial. Note that the corresponding slack matrix has $651$ distinct non-constant entry values, $230$ of them with degree $\leq 3$, in $31$ variables. Computing all products of two constraints is therefore not practical, so we resort to the above mentioned simplifications.

Under a certain labeling of the vertices, we consider the flag:
\begin{gather*}
\mc F = \{F_1 = \{\mathbf{1, 2, 3, 4, 5}, 6, 7\}, F_2 = \{\mathbf{8, 9, 10, 11, 12}, 13, 14\}, F_3 = \{1, 2, 6, 8, 14\}, \\
F_4 = \{1, 5, 8, 9, 14\}, F_5 = \{6, 8, 9, 12, 11\}, F_6 = \{1, 6, 8, 14, 9\}\},
\end{gather*}
where we ordered the vertices in each facet so that the facet is positively oriented and for the non-simplicial facets we wrote a basis in bold. We consider the corresponding reduced slack matrix with the following dehomogenization:
\begin{small}
\[
S_{\mc F}(\xx) = \begin{bmatrix}
0 & x_{1,2} & 0 & 0 & 1 & 0 \\
0 & x_{2,2} & 0 & x_{2,4} & 1 & x_{2,6} \\
0 & 1 & 1 & 1 & 1 & 1 \\
0 & x_{4,2} & x_{4,3} & x_{4,4} & 1 & x_{4,6} \\
0 & x_{5,2} & x_{5,3} & 0 & 1 & x_{5,6} \\
0 & 1 & 0 & x_{6,4} & 0 & 0 \\
0 & x_{7,2} & x_{7,3} & x_{7,4} & 1 & x_{7,6} \\
1 & 0 & 0 & 0 & 0 & 0 \\
x_{9,1} & 0 & 1 & 0 & 0 & 0 \\
1 & 0 & x_{10,3} & x_{10,4} & 1 & x_{10,6} \\
x_{11,1} & 0 & x_{11,3} & x_{11,4} & 0 & 1 \\
x_{12,1} & 0 & x_{12,3} & x_{12,4} & 0 & 1 \\
x_{13,1} & 0 & x_{13,3} & x_{13,4} & 1 & x_{13,6} \\
x_{14,1} & 0 & 0 & 0 & 1 & 0
\end{bmatrix}.
\]
\end{small}

Applying the first proposed constraint selection heuristic,  in the reconstructed slack matrix we only consider slack entries $S_{i,j}$ that avoid the vertices forming the triangle $\{2,4,7\}$, i.e., such that $i \notin \{2,4,7\}$ and $F_j$ does not contain any of $2,4$, or $7$. This leaves us with only $53$ constraints in $31$ variables. Searching for $(2,3)$-positive polynomial certificates we find the certificate of non-realizability:
\[
S_{5,2} S_{11,9} + S_{5,5} S_{13,8} + S_{5,2} S_{8,7} = 0.
\]
where the facets outside the flag used are
\begin{gather*}
F_7 = \{3, 6, 9, 10, 11\}, F_8 = \{5, 8, 9, 11, 10\} \text{ and } F_9 = \{3, 6, 8, 9, 13\}.
\end{gather*}
When rehomogenizing, we need to introduce extra variables to maintain the equality, thus obtaining the homogeneous certificate:
\[
S_{5,2}^H S_{11,9}^H {\bc x_{5,5} x_{10,5}} + S_{5,5}^H S_{13,8}^H {\bc x_{3,5} x_{10,5}} + S_{5,2}^H S_{8,7}^H {\bc x_{5,5} x_{13,5}} = 0.
\]

We found the same certificate using the second constraint selection method, selecting only pairs $(i, F_j)$ where $\{8,9,11\} \subseteq \{i\} \cup F_j$. We thus consider $43$ constraints (plus $31$ variables) in the reconstructed slack matrix.

Finally, using monomial simplification, we find a simpler certificate choosing $(k,l) = (1,2)$. In this case we consider $50$  constraints (plus $31$ variables) in the reconstructed slack matrix. The certificate is:
\begin{gather*}
S_{8,7} + S_{11,9} + S_{6,10} = 0,\\
S_{8,7}^H {\bc x_{13,5}} + S_{11,9}^H {\bc x_{10,5}} + S_{6,10}^H {\bc x_{3,5}}= 0,
\end{gather*}
where the additional column used corresponds to the set $F_{10} = \{8, 9, 10, 11, 13\}$, which comes from choosing a different facet basis for $\{7, 8, 9, 10, 11, 12, 13\}$.
\end{example}

\subsection{Numerical results}

In this section we provide performance results from our algorithm applied to a set of examples, obtained from several literature sources, of non-realizable polytopes or polytopes whose realizability is open. All computations in this paper were performed on a desktop computer with 8 cores, Intel i7-7700, running at 3.6GHz with 32GB RAM.

\medskip
\noindent\textbf{Database of simplicial spheres}

In Table \ref{tab.Database} we summarize the results about some $4$-dimensional simplicial spheres. For some of them we could not find a non-realizability certificate.

\begingroup
\renewcommand\arraystretch{1.1}
\begin{longtable}{| c | c | c | c | >{\centering\arraybackslash}m{2cm} | >{\centering\arraybackslash}m{2cm} |}
  \hline
  \centering \textbf{Name} & \textbf{\# vertices} & \textbf{\# facets} & $(k,l)$ & \textbf{\# terms in} \newline \textbf{certificate} & \textbf{Previous} \newline \textbf{certificate?} \\\hline
  Altshuler $N^{10}_{3574}$ \cite{A77} & $10$ & $35$ & $(2,2)$ & $5$ & yes \cite{JR07} \\ \hline

  Doolittle 1 \cite{D21} & $11$ & $44$ & $(2,2)$ & $6$ & no \\ \hline

  Doolittle 2 \cite{D21} & $13$ & $65$ & $(2,2)$ & $10$ & no \\ \hline

  Doolittle 3 \cite{D21} & $13$ & $65$ & $(2,3)$ & $15$ & no \\ \hline

  Novik-Zheng $\Delta^3_6$ \cite{NZ20} & $12$ & $48$ & $(2,4)$ & $10$ & yes \cite{pfeifle2020positive} \\ \hline

  Novik-Zheng $\Delta^3_n$ ($n \geq 7$) \cite{NZ20} & $2n$ & $2n(n-2)$ & $(2,4)$ & $10$ & yes \cite{pfeifle2020positive} \\ \hline \hline

  Firsching F374225 \cite{F20} & $12$ & $54$ & \multicolumn{2}{c |}{not found} & no \\ \hline

  Firsching T2775 \cite{F20} & $14$ & $49$ & \multicolumn{2}{c |}{not found} & no \\ \hline

  Zheng \cite{Z20} & $16$ & $80$ & \multicolumn{2}{c |}{not found} & yes \cite{pfeifle2020positive} \\ \hline

\caption{Database of simplicial spheres} \label{tab.Database}
\end{longtable}
\endgroup

We have seen a non-realizability certificate for Altshuler's sphere $N^{10}_{3574}$ in Example \ref{E.N10_3574}. Modifying $N^{10}_{3574}$, Joseph Doolittle recently constructed three $4$-dimensional simplicial spheres whose realizability was not known \cite{D21}. We proved that they are all non-realizable and presented an explicit certificate for one of them in Example \ref{ex:Doolittle}. The certificates for the other two spheres can be found at \cite{dataGMW21}.

\medskip
Then we consider Jockusch’s family of simplicial $3$-spheres, $\Delta^3_n$, for $n \geq 6$, whose construction is described in \cite{NZ20}. Using our algorithm, we recover \cite[Theorem 5.11]{pfeifle2020positive}.

\begin{theorem}
For $n \geq 6$, $\Delta^3_n$ is not polytopal.
\end{theorem}

\begin{proof}
Let $P$ be the simplicial sphere $\Delta^3_6$ from \cite{NZ20}. It is $4$-dimensional, has $12$ vertices and $48$ facets. Under a certain labeling of the vertices, we consider the flag:
\begin{gather*}
\mc F = \{F_1 = \{2, 8, 12, 10\}, F_2 = \{2, 3, 7, 9\}, F_3 = \{2, 4, 12, 8\}, F_4 = \{2, 3, 8, 7\}, F_5 = \{1, 6, 9, 8\} \},
\end{gather*}
where we ordered the vertices in each facet so that the facet is positively oriented. We consider the corresponding reduced slack matrix
with the following dehomogenization:
\begin{small}
\[
S_{\mc F}(\xx) = \begin{bmatrix}
x_{1,1} & 1 & x_{1,3} & x_{1,4} & 0 \\
0 & 0 & 0 & 0 & 1 \\
x_{3,1} & 0 & x_{3,3} & 0 & 1 \\
1 & 1 & 0 & 1 & 1 \\
x_{5,1} & x_{5,2} & x_{5,3} & x_{5,4} & 1 \\
x_{6,1} & 1 & x_{6,3} & x_{6,4} & 0 \\
x_{7,1} & 0 & x_{7,3} & 0 & 1 \\
0 & 1 & 0 & 0 & 0 \\
1 & 0 & x_{9,3} & x_{9,4} & 0 \\
0 & x_{10,2} & x_{10,3} & x_{10,4} & 1 \\
x_{11,1} & x_{11,2} & 1 & x_{11,4} & 1 \\
0 & x_{12,2} & 0 & x_{12,4} & 1 \\
\end{bmatrix}.
\]
\end{small}

Searching for $(2,4)$-positive polynomial certificates, we find the certificate of non-realizability:
\begin{gather*}
S_{(2,7)} S_{(8,6)} + S_{(2,7)} S_{(8,11)} + S_{(5,3)} S_{(7,9)} + S_{(2,8)} S_{(8,9)} + S_{(2,7)} S_{(6,3)} +\\
S_{(2,8)} S_{(8,10)} + S_{(2,8)} S_{(8,6)} + S_{(2,7)} S_{(8,10)} + S_{(9,4)} S_{(2,11)} + S_{(2,7)} S_{(8,9)} = 0,
\end{gather*}
where the facets outside the flag used are
\begin{gather*}
F_6 = \{2, 3, 6, 4\}, F_7 = \{3, 4, 7, 5\}, F_8 = \{3, 5, 7, 12\}, F_9 = \{2, 3, 5, 6\},\\
F_{10} = \{3, 4, 5, 6\} \text{ and } F_{11} = \{4, 5, 6, 12\}.
\end{gather*}

The above certificate only uses the facets of $\Delta^3_6$ that avoid a certain $3$-ball, $\pm B^{3,1}_6$, contained in $\Delta^3_6$, see \cite[Proof of Thm. 5.11]{pfeifle2020positive}. By the same proof, this guarantees the non-realizability of all the $3$-spheres $\Delta^3_n$ in Jockusch's family for $n \geq 6$.
\end{proof}

Notice that the above certificate has $10$ terms, whereas Pfeifle's certificate has $11$ terms.

\begin{remark}
We ran our algorithm on each of the last three spheres in Table \ref{tab.Database} for several choices of flag and $(k,l) = (2,4)$ or $(3,2)$ considering all constraints, and with $(k,l) = (3,3)$ after removing an edge of the sphere. We did not find a non-realizability certificate in any of these attempts. This does not mean that these spheres are realizable. In fact, in \cite[Section 5.1]{pfeifle2020positive} Pfeifle found a non-realizability certificate for the last sphere in the table.
\end{remark}

\smallskip
\noindent\textbf{Prismatoids}

In \cite{CS19} Criado and Santos constructed $4093$ abstract $5$-dimensional non-$d$-step prismatoids with number of vertices between $14$ and $28$. These are examples of non-Hirsch spheres but it is not known if any of them is realizable as a convex polytope. We applied our method to several of these prismatoids with many different vertex numbers and, in all cases, we could find a $(2,3)$-positive polynomial, proving that they are non-polytopal, see \cite{dataGMW21}. For instance, in the cases with the lowest number of vertices, $14$ and $15$, we could exhaustively rule out all proposed spheres. This leads us to suspect that in fact none of the $4093$ is polytopal.

\begin{proposition}
The $40$ combinatorial prismatoids with $14$ and $15$ vertices from \cite{CS19} are not realizable as convex polytopes.
\end{proposition}

\begin{remark}
The four prismatoids with $14$ vertices have $94$ facets and in \cite{CS19} are denoted by numbers $\#1039, \#1963, \#2669$ and $\#3513$. In \cite[Section 5.2]{pfeifle2020positive}, Pfeifle found non-realizability certificates for these prismatoids, each with $5$ terms of degree $4$. For the same prismatoids we found $(2,3)$-positive polynomials, which give shorter certificates with $3$ terms of degree $4$ after rehomogenization, as seen in Example \ref{E.P3513}. The complete data about our computations can be found at \cite{dataGMW21}.
\end{remark}

For each of the $36$ prismatoids with $15$ vertices and $103$, $105$ or $107$ facets we found a\break $(2,3)$-positive polynomial certificate of non-realizability. We summarize the results in Table \ref{tab.Prismatoids15v}, where for each prismatoid we list the number of terms in the non-realizability certificate we found. Again, more details can be found at \cite{dataGMW21}.
\vspace{-4mm}
\begin{center}
\begin{table}[ht!]
\renewcommand\thetable{2}
\begin{minipage}{0.25\textwidth}
\begin{longtable}{| c | c |}
  \hline
  \centering \textbf{Name} & \textbf{\# terms} \\\hline
  0213 & 4 \\ \hline

  0247 & 8 \\ \hline

  0289 & 3 \\ \hline

  0375 & 9 \\ \hline

  0554 & 9 \\ \hline

  0572 & 3 \\ \hline

  0595 & 5 \\ \hline

  0743 & 3 \\ \hline

  0800 & 3 \\ \hline
  \end{longtable}
  \end{minipage}\begin{minipage}{0.25\textwidth}
\begin{longtable}{| c | c | c |}
  \hline
  \centering \textbf{Name} & \textbf{\# terms} \\\hline
  0821 & 3 \\ \hline

  1293 & 7 \\ \hline

  1377 & 3 \\ \hline

  1649 & 7 \\ \hline

  1682 & 3 \\ \hline

  1782 & 9 \\ \hline

  1993 & 3 \\ \hline

  2063 & 8 \\ \hline

  2146 & 5 \\ \hline
  \end{longtable}
  \end{minipage}\begin{minipage}{0.25\textwidth}
\begin{longtable}{| c | c | c |}
  \hline
  \centering \textbf{Name} & \textbf{\# terms} \\\hline
  2173 & 6 \\ \hline

  2253 & 6 \\ \hline

  2348 & 6 \\ \hline

  2363 & 5 \\ \hline

  2505 & 8 \\ \hline

  2703 & 3 \\ \hline

  2864 & 3 \\ \hline

  2870 & 6 \\ \hline

  2873 & 9 \\ \hline
  \end{longtable}
  \end{minipage}\begin{minipage}{0.25\textwidth}
\begin{longtable}{| c | c | c |}
  \hline
  \centering \textbf{Name} & \textbf{\# terms} \\\hline
  2972 & 6 \\ \hline

  3022 & 3 \\ \hline

  3202 & 5 \\ \hline

  3353 & 4 \\ \hline

  3474 & 3 \\ \hline

  3672 & 3 \\ \hline

  3784 & 9 \\ \hline

  3800 & 4 \\ \hline

  4067 & 3 \\ \hline
\end{longtable}
\end{minipage}
\caption{Prismatoids with $15$ vertices \label{tab.Prismatoids15v}}
\end{table}
\end{center}

\begin{example}
Prismatoid $\#2105$ is the largest prismatoid constructed by Criado and Santos in~\cite{CS19}, having $28$ vertices and $273$ facets, two of which are non-simplicial. The corresponding slack matrix has $6261$ distinct non-constant entry values, $910$ of which have degree $\leq 3$, in $87$ variables. Computing all products of two constraints of degree $\leq 3$ is doable and running the algorithm takes $6$ minutes and $24$ seconds.

Under a certain labeling of the vertices, we consider the flag:
\begin{gather*}
\mc F = \{F_1 = \{\mathbf{1, 2, 3, 4, 6, 5}, 7, 8, 9, 10, 11, 12, 13, 14\}, \\
F_2 = \{\mathbf{15, 16, 17, 19, 18}, 20, 21, 22, 23, 24, 25, 26, 27, 28\},\\
F_3 = \{4, 8, 17, 26, 25\}, F_4 = \{3, 8, 10, 27, 25\}, F_5 = \{4, 8, 25, 26, 27\}, F_6 = \{4, 10, 25, 27, 26\}\},
\end{gather*}
where we ordered the vertices in each facet so that the facet is positively oriented and for the non-simplicial facets we write a basis in bold. We consider the corresponding reduced slack matrix $S_{\mc F}(\xx)$ with the following dehomogenization:
\[
\setcounter{MaxMatrixCols}{30}
\begin{bmatrix}
0 & 0 & 0 & 0 & 0 & 0 & 0 & 0 & 0 & 0 & 0 & 0 & 0 & 0 & \star & 1 & \star & \star & \star & \star & \star & \star & \star & \star & 1 & \star & \star & \star \\
1 & 1 & 1 & 1 & 1 & 1 & 1 & 1 & 1 & 1 & 1 & 1 & 1 & 1 & 0 & 0 & 0 & 0 & 0 & 0 & 0 & 0 & 0 & 0 & 0 & 0 & 0 & 0 \\
\star & \star & \star & 0 & \star & \star & 1 & 0 & \star & \star & \star & \star & \star & \star & \star & \star & 0 & \star & \star & \star & \star & \star & \star & \star & 0 & 0 & 1 & \star \\
\star & \star & 0 & \star & \star & \star & 1 & 0 & \star & 0 & \star & \star & \star & \star & \star & \star & \star & \star & \star & \star & \star & \star & \star & \star & 0 & 1 & 0 & \star \\
\star & \star & \star & 0 & \star & \star & 1 & 0 & \star & \star & \star & \star & \star & \star & \star & \star & \star & \star & \star & \star & \star & \star & \star & \star & 0 & 0 & 0 & \star \\
\star & \star & \star & 0 & \star & \star & \star & 1 & \star & 0 & \star & \star & \star & \star & 1 & 1 & 1 & 1 & 1 & 1 & 1 & 1 & 1 & 1 & 0 & 0 & 0 & 1
\end{bmatrix}.
\]
The above matrix is the transpose of the dehomogenized $S_{\mc F}(\xx)$, where we denote the variable entries with a $\star$.

We apply the three constraint selection heuristics described in Section \ref{S.constraintSelection}, considering in the reconstructed slack matrix only slack entries $S_{i,j}$ that
\begin{itemize}
\item avoid the vertices forming facet $\{1, 2, 5, 16, 21\}$, i.e., such that $i \notin \{1, 2, 5, 16, 21\}$ and $F_j$ does not contain any of $1, 2, 5, 16, 21$ and
\item such that $\{25, 27\} \subseteq \{i\} \cup F_j$.
\end{itemize}
This leaves us with only $219$ constraints in $87$ variables.

\indent Searching for $(2,3)$-positive polynomial certificates we find, after $22$ seconds, the certificate of non-realizability:
\begin{gather*}
S_{26, 4} S_{25, 7} + S_{3, 5} S_{25, 8} + S_{4, 4} S_{22, 5} + S_{4, 4} S_{22, 9} + S_{4, 4} S_{28, 6} = 0,
\end{gather*}
where the facets used from outside the flag are
\begin{gather*}
F_7 = \{3, 4, 8, 22, 27\}, F_8 = \{3, 8, 10, 22, 27\}, \text{and } F_9 = \{3, 10, 25, 26, 27\}.
\end{gather*}
\end{example}

\section{Concluding remarks}
We have presented a conceptually simple algorithm for producing certificates of non-realizability for abstract polytopal spheres. We first generate a novel parametrization of the realization space, and then solve a straightforward linear program which tries to find positive polynomials in the defining ideal of the realization space. We give explicit examples of certificates found via this method, both in cases that were already known where we are also able to find simpler certificates than those obtained by previous methods, and in cases where no previous certificates were known.

The certificates we produce can easily be interpreted as classical final polynomials. Unlike many other techniques used to provide such certificates, we do not need to make any assumptions on the structure of the desired final polynomial. However, close inspection of the structure of the certificates we obtain allows us to suggest further improvements to our algorithm via several heuristics that allow us to significantly decrease the size of our search space in the case of larger spheres.
While these preliminary results are very encouraging, some questions remain open. \smallskip

{\bf Are all prismatoids non-realizable?} As mentioned before, we could easily derive a\break $(2,3)$-positive polynomial for every prismatoid that we tried. It would be interesting to explore those certificates and find a general obstruction to realizability for polytopes constructed  in \cite{CS19}. \smallskip

{\bf Can the heuristics be improved?} The proposed heuristics can lighten the computational load of the method, but are not enough to deal with some of the larger cases of interest, like Firsching's spheres. It would be important to develop smarter ways of reducing the size of the problem, to make these examples more treatable.\smallskip

{\bf How important is the choice of flag?} Our algorithm starts with an arbitrary choice of flag. It is unclear, both in theoretical and in practical terms, how that can affect its performance. In particular, one could try to see if there exists some criterion that would allow choosing particularly suitable flags as a starting point.

\section*{Acknowledgements}
The authors acknowledge the extensive use of the software \textit{SageMath} \cite{Sagemath}, \textit{Macaulay2} \cite{M2} and \textit{Gurobi} \cite{Gurobi}. They also thank Julian Pfeifle for his availability to further explain his recent results and Michael Joswig for providing some useful references.  The second author thanks Marco Macchia for his suggestions on the implementation of the algorithm in \textit{SageMath}.

\bibliographystyle{plain}
\bibliography{mybibliography}
\end{document}